\documentclass[10pt]{amsart}
\usepackage{verbatim}
\usepackage{amsmath,pxfonts,amscd,amssymb,epsfig,euscript,bm}
\usepackage{amsxtra,mathrsfs}
\usepackage{pdfsync}
\newtheorem{theorem}{Theorem}[section]
\newtheorem{proposition}{Proposition}[section]
\newtheorem{lemma}{Lemma}[section]

\theoremstyle{definition}
{}

\theoremstyle{remark} 
\newtheorem{remark}{Remark}[section]

\newcommand{\del}{\partial}

\newcommand{\delb}{\bar{\partial}}

\newcommand{\cF}{{\mathcal{F}}}
\newcommand{\pbp}{\partial \bar{\partial}}

\numberwithin{equation}{section}
\usepackage{hyperref}
\hypersetup{colorlinks,citecolor=blue,plainpages=false,hypertexnames=false}
\begin{document}
\title{The Demailly system for a direct sum of ample line bundles on Riemann surfaces}
\author{Vamsi Pritham Pingali}
\address{Department of Mathematics, Indian Institute of Science, Bangalore 560012, India}
\email{vamsipingali@iisc.ac.in}
\begin{abstract} 
We prove that a system of equations introduced by Demailly (to attack a conjecture of Griffiths) has a smooth solution for a direct sum of ample line bundles on a Riemann surface. We also reduce the problem for general vector bundles to an a priori estimate using Leray-Schauder degree theory.
\end{abstract}
\maketitle
\section{Introduction}\label{sec:intro}
A holomorphic vector bundle $E$ is said to be Hartshorne ample if $\mathcal{O}_{E^*}(1)$ is an ample line bundle over $\mathbb{P}(E^*)$. There is no unique differentio-geometric notion of positivity of curvature $\Theta$ of a smooth Hermitian metric $h$. The most natural of these notions are \emph{Griffiths positivity} ($\langle v, \sqrt{-1}\Theta v\rangle$ is a K\"ahler form for all $v\neq 0$), \emph{Nakano positivity} (the bilinear form defined by $\sqrt{-1}\Theta$ on $T^{1,0}M \otimes E$ is positive-definite), and \emph{dual-Nakano positivity} (the Hermitian holomorphic bundle $(E^*,h^*)$ is Nakano \emph{negative}). Nakano positivity and dual-Nakano positivity imply Griffiths positivity and all three of them imply Hartshorne ampleness. A famous conjecture of Griffiths \cite{Griffiths} asks whether Hartshorne ample vector bundles  admit Griffiths positively curved metrics. This conjecture is still open. However, a considerable amount of work has been done to provide evidence in its favour \cite{Bo, Camp, dem2, Liu,  MT, Pinchern, Um}.\\
\indent Relatively recently, Demailly \cite{dem2} proposed a programme to prove the aformentioned conjecture of Griffiths (in fact, this approach aims at proving dual Nakano positivity, if it works). Demailly's approach involves solving a family (depending on a parameter $0\leq t \leq 1$) of PDE that we call the Demailly system. Unfortunately\footnote{The author thanks J.-P. Demailly for this observation.}, the cotangent bundle of a compact ball quotient is ample but does not admit dual Nakano positively curved metrics. Thus, Demailly's approach cannot work in general. Nonetheless, the aforementioned counterexample does admit a dual Nakano \emph{semi-positively} curved metric. It is an interesting question to know if the Demailly system can be solved even in higher dimensions on $[0,1)$ (and perhaps blows up at $t=1$). Even if the maximal time of existence is not $t=1$, it might still be related to interesting numerical conditions on the bundle. We aim to provide a proof-of-concept for Demailly's approach by studying the simplest non-trivial case of a direct sum of ample line bundles over a compact Riemann surface $M$. We also reduce the vector bundle case on Riemann surfaces to an a priori estimate. \\
\indent Let $h_t=e^{-f_t}g_th_0$ where $\det(g_t)=1$ and $g_t>0$. In this setting, the Demailly system boils down to the following set of equations. 
\begin{gather}
\det\left(\frac{\sqrt{-1}F}{\omega_0}+(1-t)\alpha_0\right)=e^{\lambda f}a_t\nonumber \\
\sqrt{-1}F-\frac{1}{r}\mathrm{tr}(\sqrt{-1}F) =-e^{\mu f} \ln g\omega_0.
\label{eq:mainsystem}
\end{gather}
We choose $\mu=1$. Substituting $h=e^{-f} gh_0$, $F_0^0=F_0-\frac{1}{r}\omega_0$, $\Delta f = \frac{\sqrt{-1}\pbp f}{\omega_0}$ we get
\begin{gather}
\det\left(\Delta f+\frac{1}{r}-e^f \ln g +(1-t)\alpha_0\right)=e^{\lambda f}a_t\nonumber \\
\sqrt{-1}F_0^0+\sqrt{-1}\bar{\partial}(\partial g g^{-1}) =-e^{f} \ln g\omega_0.
\label{eq:mainsysteminfandg}
\end{gather}
Suppose $E=\displaystyle \oplus_{i=1}^r L_i$ where $L_i$ are holomorphic line bundles and $h_0$ is a direct sum of metrics, we can attempt to solve \ref{eq:mainsysteminfandg} using a direct sum of metrics. Thus, $\ln g = u_1 \oplus u_2 \ldots$ where $u_i$ are smooth functions satisfying 
\begin{gather}
\displaystyle \sum_i u_i=0 \nonumber \\
\det\left(\Delta f+\frac{1}{r}-e^f \ln g +(1-t)\alpha_0\right)=e^{\lambda f}a_t\nonumber \\
\sqrt{-1}(F_0^0)_i+\sqrt{-1}\delb \del u_i =-e^{f} u_i\omega_0.
\label{eq:directsummainsysteminfandg}
\end{gather}
In this paper we prove the following main result.
\begin{theorem}
If $L_i$ are ample bundles, then the system \ref{eq:mainsysteminfandg} has a smooth Griffiths-positively curved solution.
\label{thm:main1}
\end{theorem}
In the proof of this theorem, we do not assume at the outset that $L_i$ admit positively curved metrics. Moreover, as discussed in Remark \ref{rem:nontrivialityevenassumingpositivity}, even if we make such an assumption, it is not immediately clear whether the proof can be simplified significantly. \\
\indent Along the way, we also reduce the problem for general vector bundles to an a priori estimate.
\begin{theorem}
Suppose any $C^{2,\gamma}$ solution $(f,g)$ of \ref{eq:mainsysteminfandg} satisfies $f\geq -C$, where $C$ is independent of $0\leq t\leq 1$. Then there exists a smooth Griffiths-positively curved solution to \ref{eq:mainsysteminfandg}.
\label{thm:main2}
\end{theorem}
The strategy of the proofs is as follows:
\begin{enumerate}
\item In Section \ref{sec:Opennessatt=0} we choose $\lambda, a_0>>1$ so that there is a solution at $t=0$ and for nearby $t$ (depending smoothly on $t$), and uniqueness holds for $t=0$. In this section we do everything for a general vector bundle. Unfortunately, openness at general $t$ seems elusive at present (if true at all). Therefore, the method of continuity cannot be used and we resort to Leray-Schauder degree theory.
\item In Section \ref{sec:Aprioriestimates} we prove a priori estimates on any solution of \ref{eq:directsummainsysteminfandg} independent of $t$. This step appears to be very challenging for vector bundles. In particular, we do not have the analogue of Proposition \ref{prop:fundamentalestimateinthedirectsumcase}. Even if such an estimate is granted, while it is not hard to use Uhlenbeck compactness to produce a limiting connection on a limiting bundle, the issue is whether the connection is Yang-Mills (such a result does not seem to have been proven even for the usual Hermitian-Einstein metrics for the continuity path in \cite{UY}, at least to the author's knowledge). Given such a result, it is easy to use ampleness (and the Harder-Narasimhan filtration) to produce a contradiction.
\item In Section \ref{sec:Conclusion} we set up Leray-Schauder degree theory to prove the existence of a Griffiths positively curved solution to \ref{eq:directsummainsysteminfandg}. In this section, we set up the Leray-Schauder degree theory for general vector bundles to prove Theorems \ref{thm:main1} and \ref{thm:main2}. 
\end{enumerate}
\subsection*{Acknowledgements} 
This work is partially supported by grant F.510/25/CAS-II/2018(SAP-I) from UGC (Govt. of India), and a MATRICS grant MTR/2020/000100 from SERB (Govt. of India). The author thanks Ved Datar, Jean-Pierre Demailly, Richard Wentworth, Sandeep Kunnath, and Swarnendu Sil for fruitful discussions.
\section{The Demailly system near $t=0$}\label{sec:Opennessatt=0}
\indent In this section, we do not assume the direct sum ansatz \ref{eq:directsummainsysteminfandg}. At $t=0$, choose $f_0=0$ and $g_0$ to solve the second equation of \ref{eq:mainsysteminfandg} using the results of \cite{UY}. (Note that in the case of a direct sum, standard elliptic theory enables us to choose $g_0$ to be a direct sum. By uniqueness \cite{PinDem}, this is the only solution in such a case.) Choose $\alpha_0>1$ so that $\frac{1}{r}+\alpha_0-\ln(g_0)>0$. Choose $a_t=a_0=\det(\frac{1}{r}+\alpha_0-\ln(g_0))$ for all $0\leq t \leq 1$. Note that $\alpha_0, a_0$ depend on $g_0$ but not the other way around. Moreover, $g_0, \alpha_0, a_0, f_0$ do \emph{not} depend on $\lambda$.
\begin{proposition}
If $\alpha_0, \lambda$ are large, then there exists a smooth solution $f_t, g_t$ to \ref{eq:mainsysteminfandg} on $t\in [0,t_0)$ where $t_0$ depends on $\lambda, g_0, \alpha_0$. Moreover, $f_t, g_t$ depend smoothly on $t$ and are locally unique, i.e., if $f,g$ are sufficiently close (depending only on $\lambda, g_0, \alpha_0$) to $f_t,g_t$ and solve \ref{eq:mainsysteminfandg}, then $f=f_t, g=g_t$. If $E$ and $h_0$ are direct sums as in \ref{eq:directsummainsysteminfandg}, then $g$ is a direct sum too.
\label{prop:Opennessat0}
\end{proposition}
\begin{proof}
Consider the map $T(f,g',t)=(T_1(f,g',t),T_2(f,g',t))$ where $0\leq t \leq 1$, $f,g'$ are smooth and satisfy $\det(g')=1, g'g_0>0$, $\Delta f+\frac{1}{r}-e^f \ln (g'g_0) +(1-t)\alpha_0>0$,  and
\begin{gather}
T_1(f,g',t)=\ln\left(\det\left(\Delta f+\frac{1}{r}-e^f \ln (g'g_0) +(1-t)\alpha_0\right)\right)-\lambda f-\ln(a_0) \nonumber \\
T_2(f,g',t)=\sqrt{-1}F_0^0+\sqrt{-1}\bar{\partial}(\partial (g'g_0) g_0^{-1}g'^{-1})+e^{f} \ln (g' g_0)\omega_0.
\label{eq:defofT}
\end{gather}
Let $0<\gamma<1$. The map $T$ extends to a Banach submanifold of $C^{2,\gamma} \times C^{2,\gamma} \times [0,1]$. We shall prove that $D_{f,g'}T(0,I,0): C^{2,\gamma}\times C_0^{2,\gamma} \rightarrow C^{0,\gamma}\times C^{0,\gamma}$ (where $C_0^{2,\gamma}$ consists of $C^{2,\gamma}$ endomorphisms $\delta g'$ such that $\mathrm{tr}(\delta g')=0$) is an isomorphism. Indeed, 
\begin{gather}
 D_{f,g'}T_1(0,I,0) [\delta f, \delta g']= \mathrm{tr}\left(\Big(\frac{1}{r}-\ln g_0+\alpha_0\Big )^{-1} (\Delta \delta f-\delta f \ln g_0 - \delta \ln (g' g_0)) \right)-\lambda\delta f\nonumber \\
\langle D_{f,g'}T_2(0,I,0) [\delta f, \delta g'] \rangle=\sqrt{-1}\bar{\partial} \partial_0 \delta g' +\delta f\ln g_0 \omega_0 +\delta \ln (g' g_0)\omega_0.
\label{eq:linearisationfirstequation}
\end{gather}
At this juncture, consider
\begin{gather}
\left\langle \left [DT[\delta f, \delta g'], \Big(\mathrm{tr}\Big(-\frac{1}{r}+\ln g_0-\alpha_0\Big )^{-1}\Big)^{-1}\delta f \omega_0, \delta g'\right ]\right\rangle \nonumber \\ = \displaystyle \int_M \Bigg( \vert \nabla \delta f \vert^2
+\vert \partial_0 \delta g' \vert^2 +\left(\frac{\lambda}{\mathrm{tr}\Big(\frac{1}{r}-\ln g_0+\alpha_0\Big )^{-1}}+\frac{\mathrm{tr}\left(\Big(\frac{1}{r}-\ln g_0+\alpha_0\Big )^{-1}\ln g_0\right)}{\mathrm{tr}\Big(\frac{1}{r}-\ln g_0+\alpha_0\Big )^{-1}} \right)\delta f^2 \nonumber \\
+ \mathrm{tr}\left((\delta g')^{\dag} \delta \ln(g'g_0)\right)+\delta f\frac{\mathrm{tr}\left(\Big(\frac{1}{r}-\ln g_0+\alpha_0\Big )^{-1}\delta\ln (g'g_0)\right)}{\mathrm{tr}\Big(\frac{1}{r}-\ln g_0+\alpha_0\Big )^{-1}}+\delta f \mathrm{tr}((\delta g')^{\dag} \ln(g_0))\Bigg ).
\label{eq:symmetricfunctional}
\end{gather}
By the Geometric-Mean-Harmonic-Mean inequality, for sufficiently large $\lambda, \alpha_0$, 
\begin{gather}
\left\langle \left [DT[\delta f, \delta g'], \Big(\mathrm{tr}\Big(-\frac{1}{r}+\ln g_0-\alpha_0\Big )^{-1}\Big)^{-1}\delta f \omega_0, \delta g'\right ]\right\rangle\nonumber \\
\geq \displaystyle \int_M \Bigg( \frac{\lambda}{2} \delta f^2
+ \mathrm{tr}\left((\delta g')^{\dag} \delta \ln(g'g_0)\right)+\delta f\frac{\mathrm{tr}\left(\Big(\frac{1}{r}-\ln g_0+\alpha_0\Big )^{-1}\delta\ln (g'g_0)\right)}{\mathrm{tr}\Big(\frac{1}{r}-\ln g_0+\alpha_0\Big )^{-1}}\nonumber \\
+\delta f \mathrm{tr}((\delta g')^{\dag} \ln(g_0))\Bigg ).
\label{ineq:afterGMHM}
\end{gather}
The proof of Lemma 2.1 in \cite{UY} shows that when $g'=I$,  $$C\mathrm{tr}\left((\delta g')^{\dag} \delta \ln(g'g_0)\right)\geq  \mathrm{tr}((\delta g')^{\dag} \delta g'), $$ for some $C$ depending only on $g_0$, and that $$\mathrm{tr}\left((\delta g')^{\dag} \delta \ln(g'g_0)\right)\geq  \mathrm{tr}((\delta g')^{\dag} \delta g')\geq \mathrm{tr}((\delta\ln (g' g_0))^2).$$
Thus, using Cauchy's inequality, we see that for sufficiently large $\lambda$, there exists a constant $C>0$ (depending only on $g_0$) such that
\begin{gather}
\left\langle \left [DT[\delta f, \delta g'], \Big(\mathrm{tr}\Big(-\frac{1}{r}+\ln g_0-\alpha_0\Big )^{-1}\Big)^{-1}\delta f \omega_0, \delta g'\right ]\right\rangle\geq \displaystyle \int_M \left(\delta f^2+\frac{1}{C}\Vert \delta g'\Vert^2 \right).
\end{gather}
Hence $Ker(DT)=Ker(DT^{*})=0$ and therefore, elliptic theory implies that $DT$ is an isomorphism. Hence, the infinite-dimensional implicit function theorem implies that $f_t, g_t$ exist (solving the equation) and are locally unique. \\
\indent In the direct sum case, exactly the same arguments go through for \ref{eq:directsummainsysteminfandg} word-to-word.
\end{proof}
We can prove a stronger uniqueness result.
\begin{proposition}
There exists a constant $b>r>0$ such that for every $\alpha_0, \lambda \geq b$,  \ref{eq:mainsysteminfandg} has a unique smooth solution at $t=0$ satisfying $\frac{\sqrt{-1}F}{\omega_0}+\alpha_0>0$.
\label{prop:Uniquenessat0}
\end{proposition}
\begin{proof}

We need to prove that $(f=0,g=g_0)$ is the only smooth solution at $t=0$ satisfying $\frac{\sqrt{-1}F}{\omega_0}+\alpha_0>0$. Suppose $(f,g)$ is another such solution. Thus 
\begin{gather}
\Delta f+\frac{1}{r}-e^f \ln g+\alpha_0>0, \nonumber \\
\frac{1}{r}-\ln g_0 +\alpha_0>0.
\label{eq:coneinuniqueness}
\end{gather}
 The two solutions can be connected by a path $f_s=sf, g_s=e^{s\ln(g)+(1-s)\ln(g_0)}$. Note that $\frac{df_s}{ds}=f$ and $\frac{dg_s}{ds}=Dexp_{s\ln(g)+(1-s)\ln(g_0)}(\ln(g)-\ln(g_0))$ where $D\exp_A (B)$ is the derivative map of the exponential function at the endomorphism $A$ acting (linearly) on the endomorphism $B$. We want to conclude that $\Delta f_s+\frac{1}{r}-e^{f_s}\ln(g_s)+\alpha_0$ is positive (and in fact, can be made arbitrarily large by increasing $\alpha_0, \lambda$). To this end, we need some a priori estimates.
\begin{lemma}
As $\alpha_0\rightarrow \infty$, $e^{\lambda f}\rightarrow 1$, $\vert ln g\vert\leq C$, and $\Delta f=o(\alpha)$.
\label{lem:aprioriestimatesforuniqueness}
\end{lemma}
\begin{proof}
The equations satisfied by $f,g$ are:
\begin{gather}
\det\left(\Delta f+\frac{1}{r}-e^f \ln (g)+\alpha_0\right ) = e^{\lambda f}\det\left ( \frac{1}{r}-\ln(g_0)+\alpha_0\right )\nonumber \\
\sqrt{-1}F^0_0+\sqrt{-1}\bar{\partial} (\partial g g^{-1})=-e^f \ln (g) \omega_0.
\label{eq:satisfiedbyfandguniqueness}
\end{gather}
At this point we recall a useful lemma (Lemma 2.4 in \cite{UY}) of Uhlenbeck and Yau (whose proof goes through verbatim in spite of $\epsilon=e^f$ not being a constant).
\begin{lemma}
Let $u=\ln(g)$ and $\vert u \vert^2=tr(u^2)$. Then 
\begin{gather}
e^f \vert u\vert^2-\frac{1}{2}\Delta \vert u \vert^2\leq \vert u \vert \left \vert \frac{\sqrt{-1}F_0^0}{\omega_0} \right \vert.
\label{eq:UYestimate}
\end{gather}
\label{lem:UYlemmaforestimate}
\end{lemma}
At the maximum of $\vert u\vert$, $e^{f_{min}}\vert u \vert_{max} \leq e^f \vert u \vert\leq  \left \vert \frac{\sqrt{-1}F_0^0}{\omega_0} \right \vert\leq C_1$. At the minimum of $f$, $\Delta f +\frac{1}{r}-e^{f_{min}} \ln(g) +\alpha_0 \geq \alpha_0+\frac{1}{r}-C_1>0$ if $\alpha_0$ is large enough. Thus, 
\begin{gather}
e^{\lambda f_{\min}} \geq \frac{\det(\alpha_0+\frac{1}{r}-C_1)}{\det(\alpha_0+\frac{1}{r}-\ln g_0)} \rightarrow 1 \ \mathrm{as} \ \alpha_0 \rightarrow \infty
\label{ineq:lowerboundonfforuniqueness}
\end{gather}
Therefore, $\vert u \vert \leq C_1\frac{\det(\alpha_0+\frac{1}{r}-\ln g_0)}{\det(\alpha_0+\frac{1}{r}-C_1)}\leq C_2$. Moreover, at the maximum of $f$, $\Delta f\leq 0$ and hence 
\begin{gather}
e^{\lambda f} a_0 \leq \det\left (\frac{1}{r}-e^{f}\ln(g)+\alpha_0 \right ) \nonumber \\
\Rightarrow e^{\lambda f_{\max}}\leq \frac{\det\left (\frac{1}{r}-e^{f_{max}}\ln(g)+\alpha_0 \right )}{\det\left (\frac{1}{r}-\ln(g_0)+\alpha_0 \right )} 
\label{ineq:upperboundonf}
\end{gather}
We claim that as $\alpha_0 \rightarrow \infty$, $e^{f_{max}}\leq C_3$ for some $C_3>0$. Indeed, if $e^{f_{max}} \rightarrow \infty$, then $\frac{e^{f_{max}} \vert ln(g) \vert}{\alpha_0}\rightarrow \infty$. But the left-hand-side goes to $\infty$ faster because $\lambda>r$ and $\vert \ln(g) \vert\leq C_2$. Hence we have a contradiction and $e^{f_{max}} \leq C_3$. Now returning to Inequality \ref{ineq:upperboundonf}, as $\alpha_0\rightarrow \infty$, the right-hand-side approaches $1$. Hence, 
\begin{gather}
1-O(\frac{1}{\alpha_0}) \leq e^{\lambda f} \leq 1+O(\frac{1}{\alpha_0})
\label{ineq:boundsonfbothsidesuniqueness}
\end{gather}
\begin{remark}
Note that by applying the Arithmetic-Mean-Geometric-Mean inequality and the assumption $\mathrm{tr}(\ln(g)) =0$ to the right-hand-side of Inequality \ref{ineq:upperboundonf}, we see that $e^{\lambda f}$ is bounded above (albeit with a bound depending on $\alpha_0$) even for time $t$ solutions (as opposed to $t=0$). This observation shall be useful later on.
\label{rem:upperboundonf}
\end{remark}
Now divide by $\alpha_0^r$ on both sides of the first equation in \ref{eq:satisfiedbyfandguniqueness}. As $\alpha_0\rightarrow \infty$, the right-hand-side approaches $1$ and any limit of the left-hand-side is $\det\left( 1+ \lim_{\alpha_0\rightarrow \infty} \frac{\Delta f}{\alpha_0}\right )$. Hence, 
\begin{gather}
\lim_{\alpha_0\rightarrow \infty} \frac{\vert \Delta f \vert}{\alpha_0} \rightarrow 0.
\label{eq:limitofDeltaf}
\end{gather}
\end{proof}
Therefore, 
\begin{gather}
\Delta f_s+\frac{1}{r}-e^{f_s}\ln(g_s)+\alpha_0=\alpha_0+o(\alpha).
\label{eq:uniquenesspathconeasymptotics}
\end{gather}
Using the same notation as in the proof of Proposition \ref{prop:Opennessat0}, 
\begin{gather}
0=T(f,g',0)-T(0,I,0)=\int_0^1 \frac{dT(f_s,g'_s=g_sg_0^{-1},0)}{ds}ds\nonumber\\
=\int_0^1 DT_{f,g'}(f_s,g'_s,0)\left[\frac{df_s}{ds},\frac{dg'_s}{ds}\right] ds\nonumber\\
= (\int_0^1\frac{dT_1}{ds}ds, \int_0^1\frac{dT_2}{ds}ds).
\label{eq:DifferenceofT}
\end{gather}
We compute the first component as follows.
\begin{gather}
\int_0^1\frac{dT_1}{ds}ds = \cF_1 \Delta f-\mathrm{tr}(\cF_2(g-g_0))-\cF_3 f-\lambda f,
\label{eq:firstcomponentofderivativeofTuniqueness}
\end{gather}
where
\begin{gather}
\cF_1=\int_0^1 \mathrm{tr}\left(\left(\Delta f_s+\frac{1}{r}-e^{f_s}\ln g_s +\alpha_0\right)^{-1} \right) ds \nonumber \\
\cF_2=\int_0^1  \left(\Delta f_s+\frac{1}{r}-e^{f_s}\ln g_s +\alpha_0\right)^{-1} e^{f_s} ds \nonumber \\
\cF_3=\int_0^1 e^{f_s}\mathrm{tr}\left(\left(\Delta f_s+\frac{1}{r}-e^{f_s}\ln g_s +\alpha_0\right)^{-1} \ln(g_s) \right) ds. 
\label{eq:defofF1F2etc}
\end{gather}
Now we compute the second component.
\begin{gather}
\frac{dT_2}{ds}=\int_0^1 \sqrt{-1}\bar{\partial}\left(\partial \left(\frac{dg_s}{ds}\right)g_s^{-1}-\partial (g_s) g_s^{-1}\frac{dg_s}{ds} g_s^{-1} \right) ds +G_1 f \omega_0+ G_2 (g-g_0) \omega_0, \nonumber\\
=\int_0^1 \sqrt{-1}\bar{\partial}\partial_s \left(\frac{dg_s}{ds}g_s^{-1}\right) ds +G_1 f \omega_0+ G_2 (g-g_0) \omega_0
\label{eq:secondcomponentofderivativeofTuniqueness}
\end{gather}
where 
\begin{gather}
G_1=\int_0^1 e^{f_s} \ln(g_s)ds \nonumber \\
G_2=\int_0^1 e^{f_s}ds. 
\label{eq:defofG1G2etc}
\end{gather}
Now we multiply Equation \ref{eq:DifferenceofT} and integrate-by-parts to obtain the following.
\begin{gather}
0=\langle \left(\int_0^1 \frac{dT_1}{ds}ds, \int_0^1 \frac{dT_2}{ds}ds \right), \left(-\cF_1^{-1} f, \int_0^1 \frac{dg_s}{ds}g_s^{-1}ds\right) \rangle \nonumber \\
\geq \displaystyle \int_M \Bigg (\cF_1^{-1}ftr(\cF_2(g-g_0))+\cF_1^{-1}\cF_3f^2+\cF_1^{-1}\lambda f^2+ftr\left(G_1 \int_0^1 \frac{dg_s}{ds}g_s^{-1}ds\right)\nonumber \\
+G_2tr\left(\int_0^1 \frac{dg_s}{ds}g_s^{-1} (g-g_0)ds\right)  \Bigg).
\label{ineq:multiplyintegratebyparts}
\end{gather}
For large $\alpha_0$, Lemma \ref{lem:aprioriestimatesforuniqueness} can be used to show easily that 
\begin{gather}
0\geq \displaystyle \int_M \Bigg (-C\vert f\vert \vert g-g_0\vert +\frac{\lambda}{2} f^2
+G_2tr\left(\int_0^1 \frac{dg_s}{ds}g_s^{-1} (g-g_0)ds\right)  \Bigg).
\label{ineq:uniquenessafterestimates}
\end{gather}
Using Lemma 2.1 in \cite{UY} and Lemma \ref{lem:aprioriestimatesforuniqueness} we see that (for a larger constant $C$)
\begin{gather}
0\geq \displaystyle \int_M \Bigg (-C\vert f\vert \vert g-g_0\vert +\frac{\lambda}{2} f^2
+\frac{1}{C}\vert g-g_0 \vert^2  \Bigg).
\label{ineq:finalinequalityforuniqueness}
\end{gather}
Therefore $f=0=g-g_0$ if $\lambda$ is sufficiently large.
\end{proof}
For the remainder of this paper, we fix $\lambda, \alpha_0 >>1$ so that Propositions \ref{prop:Opennessat0} and \ref{prop:Uniquenessat0} are applicable.
\section{A priori estimates}\label{sec:Aprioriestimates}
\indent We shall prove a priori estimates to $C^{2,\gamma}$ solutions of \ref{eq:directsummainsysteminfandg} in this section. However, not all that follows is restricted to the direct sum case. We denote constants independent of $t$ by $C$. In particular, unless specified otherwise $C$ may vary from line to line.\\
\indent By Remark \ref{rem:upperboundonf} we see that 
\begin{gather}
e^{\lambda f}\leq C. 
\label{ineq:upperboundonfforapriori}
\end{gather}
Recall the Green representation formula:
\begin{gather}
v(x)=\fint v\omega_0 + \int G(x,y) \Delta_y v \omega_0,
\label{eq:Green}
\end{gather}
where $C\ln(d(x,y)^2)\leq G(x,y) \leq 0$. Since $\Delta f \geq -C$, we see that 
\begin{gather}
f-\fint f\leq C.
\label{ineq:boundonfminusfint}
\end{gather} 
Note that if $f\geq -C$, using Lemma \ref{lem:UYlemmaforestimate} we see that $\vert u\vert \leq C$. Hence, $\vert \Delta f \vert\leq C$ and therefore $\Vert f\Vert_{C^{2,\gamma}}\leq C$. Multiplying both sides of the second equation in \ref{eq:mainsysteminfandg} by $g$, taking trace and integrating we get
\begin{gather}
\displaystyle \int_M (\vert \nabla g \vert^2 +\vert \partial g \sqrt{g}^{-1} \vert^2)\leq C \nonumber\\
\Rightarrow \Vert \nabla g \Vert_{L^2}\leq C.
\label{ineq:estimateonsobolevofginthebestcase}
\end{gather} 
Let $A=\del g g^{-1}$. Since $\vert \delb A \vert_{L^2}\leq C$ and $\Vert A \Vert_{L^2} \leq C$ by \ref{ineq:estimateonsobolevofginthebestcase}, by elliptic theory, $\Vert A\Vert_{L^p} \leq \Vert A \Vert_{W^{1,2}}\leq C$. Returning back to \ref{eq:mainsysteminfandg} we see that $\Vert g \Vert_{C^{1,\gamma}}\leq\Vert g \Vert_{W^{2,p}}\leq C$ and hence $\Vert g \Vert_{C^{2,\gamma}}\leq C$. While we cannot explicitly solve for $\Delta f$, we can do so implicitly. Let $\mathcal{A}$ be the submanifold of $\mathbb{R}\times GL(r,\mathbb{C})$ consisting of $(u,A)$ such that $\det(A)=1$, $A$ is Hermitian and positive-definite, and 
\begin{gather}
u+A>0.
\label{ineq:coneconditionforA}
\end{gather}
Define a smooth map $L(v,A): \mathcal{A}\rightarrow (0,\infty)$ as $$L(v,A)=\det(v+A).$$ It is then easy to prove the following lemma.
\begin{lemma}
For each fixed $A$, the map $L_{A}=L(,A)$ is a diffeomorphism. Moreover, $L_A^{-1}(\eta)$ depends smoothly on $(A,\eta)$.
\label{lem:implicitsolutionforDeltaf}
\end{lemma}
Thus the second equation is of the form 
\begin{gather}
\Delta f = L_{-e^f \ln(g)+\frac{1}{r}+\alpha_0(1-t)}^{-1} (e^{\lambda f}a_0).
\label{eq:solvingforDeltaf}
\end{gather} 
Now we can bootstrap to get smoothness and estimates of any order. \\
\indent We now reduce the estimate $f\geq -C$ to a another estimate. This proposition is restricted to the direct sum case \ref{eq:directsummainsysteminfandg} and uses the assumption that $L_i$ are ample.
\begin{proposition}
Assume that 
\begin{gather}
\Vert \ln(g) e^f \Vert_{C^0} \leq C.
\label{ineq:fundamentalestimate}
\end{gather}
Then there exists a $C$ independent of $t$ such that if  $(f,u_1,u_2,\ldots)$ is a $C^{2,\gamma}$ tuple of real-valued functions satisfying \ref{eq:directsummainsysteminfandg} and $\Delta f+\frac{1}{r}-e^f \ln(g)+(1-t)\alpha_0>0$, then $f\geq -C$.
\label{prop:Apriori}
\end{proposition}
\begin{proof}
Suppose not. Then there exists a sequence of times $t_n$, points $p_n$, and solutions $(f_n,g_n)$ such that $f_n(p_n)=\min f_n \rightarrow -\infty$. We ignore the subscript $n$ for most of the remainder of the proof. \\
\indent Clearly, $\vert \Delta f \vert\leq C$. Hence $\Vert f-\fint f \Vert_{W^{2,p}}\leq C$ and $f\rightarrow -\infty$ uniformly. Denote by $f_{\infty}$ the weak zero-average $W^{2,p}$ limit (which is a strong $C^{1,\gamma}$ limit) of $f-\fint f$. Likewise, $\Vert u_i-\fint u_i \Vert_{W^{2,p}}\leq C$. We can bootstrap this estimate to get $\Vert u_i-\fint u_i \Vert_{C^{2,\gamma'}}\leq C$ and can assume that $u_i-\fint u_i \rightarrow u_{i,\infty}$ in $C^{2,\gamma}$. Moreover, $u_ie^f=(u_i-\fint u_i)e^{f}+\fint u_i e^{f-\fint f} e^{\fint f}$. The first term approaches zero and the second approaches $C_{i,\infty}e^{f_{\infty}}$ where $C_{i,\infty}=\lim \fint u_i e^{\fint f}$. Since $\det(\Delta f-e^f \ln(g) +\frac{1}{r}+(1-t)\alpha_0)\rightarrow 0$, we see that in the limit $\Delta f_{\infty}=C_{i,\infty}e^{f_{\infty}}-\frac{1}{r}-(1-t)\alpha_0$ weakly for the $i$ for which $C_{i,\infty}$ is maximum (and hence positive because $\sum_i u_i=0$). Thus $f_{\infty}$ is a constant and 
\begin{gather}
\int C_{i,\infty} e^{f_{\infty}}=\frac{c_1(E)}{r}+(1-t)\alpha_0 c_1(E) \geq \frac{c_1(E)}{r}.
\label{ineq:firstforcon}
\end{gather}
 However, 
\begin{gather}
\int u_i e^f \omega_0 = \int C_{i,\infty} e^{f_{\infty}}\omega_0=-\int (F^0_0)^i=\frac{c_1(E)}{r}-c_1(L_i)<\frac{c_1(E)}{r},
\label{eq:secondforcon}
\end{gather}
where the ampleness of $L_i$ was used. We have a contradiction from Equations \ref{ineq:firstforcon} and \ref{eq:secondforcon}.
\end{proof}
\indent We now prove the following proposition.
\begin{proposition}
Any solution $(u_i, f)$ of \ref{eq:directsummainsysteminfandg} satisfies $$\sum_i \Vert u_i e^f \Vert_{C^0}\leq C.$$
\label{prop:fundamentalestimateinthedirectsumcase}
\end{proposition}
\begin{proof}
Recall that $C\sum_i \vert u_i \vert e^f +C \geq \Delta f \geq \sum_i \frac{\vert u_i \vert e^f}{C} -C$. Hence, for all $i$
\begin{gather}
\Delta (u_i+C f)\geq -C,
\label{eq:plusforeachi}
\end{gather}
and 
\begin{gather}
\Delta(u_i-C f)\leq C.
\label{eq:minusforeachi}
\end{gather}
Applying \ref{eq:Green} to \ref{eq:plusforeachi} and \ref{eq:minusforeachi} we get 
\begin{gather}
u_i+Cf \leq \fint (u_i+Cf) +C \nonumber \\
u_i-Cf \geq \fint (u_i-Cf) -C \nonumber \\
\Rightarrow \vert u_i-\fint u_i \vert \leq C\vert f -\fint f \vert+C.
\label{ineq:estimateonui}
\end{gather}
Moreover, at the maximum of $f$, $\vert u_i \vert e^{f_{max}}\leq C$. Thus $\vert \fint u_i \vert \leq C\vert f-\fint f \vert+C+Ce^{-f_{max}}$. Therefore, $$\vert u_i \vert e^f \leq \vert u_i-\fint u_i \vert e^f +\vert \fint u_i \vert e^f \leq Ce^f +C\vert f-\fint f \vert e^f + Ce^{f-f_{max}}\leq C,$$ using \ref{ineq:upperboundonfforapriori} and \ref{ineq:boundonfminusfint}.
\end{proof}
\indent Combining Propositions \ref{prop:Apriori} and \ref{prop:fundamentalestimateinthedirectsumcase} with the arguments made earlier, we see that $f,g$ are smooth and satisfy a priori estimates of all orders in the direct sum case.
\begin{remark}
Naively, one might think that since $L_i$ admit positively curved metrics, all the a priori estimates must be trivial. However, to get a lower bound on $f$, if one attempts to use the maximum principle, one would need $\frac{1}{r}+(1-t)\alpha_0-e^f u_i >\frac{1}{C}$. Unfortunately, Proposition \ref{prop:fundamentalestimateinthedirectsumcase} is not effective and hence we are stymied. 
\label{rem:nontrivialityevenassumingpositivity}
\end{remark}
\section{Conclusion of the proof}\label{sec:Conclusion}
\indent We use Leray-Schauder degree theory to complete the proof of Theorems \ref{thm:main1} and \ref{thm:main2}. Let $\mathcal{B}$ be the bounded Banach submanifold of $C^{2,\gamma}\times C^{2,\gamma}\times [0,1]$ consisting of $(f,g,t)$ such that $\det(g)=1$, $g$ is $h_0$-Hermitian and positive-definite, $\Delta f+\frac{1}{r}-e^f \ln(g)+(1-t)\alpha_0>0$, and $\Vert f \Vert_{C^{2,\gamma}}+\Vert \ln(g) \Vert_{C^{2,\gamma}}< 2C$ and $\Delta f-e^f \ln(g)+\frac{1}{r}+\alpha_0(1-t)>\frac{1}{2C}$ where $C$ is the a priori estimate for $\Vert f \Vert_{C^{2,\gamma}}+\Vert \ln(g)\Vert_{C^{2,\gamma}}$ and $\frac{1}{C}$ is the lower bound for $\Delta f-e^f \ln(g)+\frac{1}{r}+\alpha_0(1-t)$ from section \ref{sec:Aprioriestimates}. Given $(f,g,t) \in \overline{\mathcal{B}}$ define the $C^{2,\gamma}$ pair $U,V$ (with $\det(V)=1$) by 
\begin{gather}
\Delta U-e^f \ln(g)+\frac{1}{r}+\alpha_0(1-t)>0,\nonumber \\
\Delta U =L^{-1}_{-e^f \ln(g)+\frac{1}{r}+\alpha_0(1-t)}(e^{\lambda U}a_0),
\label{eq:defofuforLS}
\end{gather}
and
\begin{gather}
F^0_0+\delb (\del V V^{-1})=-e^f\ln(V)\omega_0,
\label{eq:defofvforLS}
\end{gather}
in the vector bundle case. In the direct sum case, exactly the same definitions are used except that $V$ is assumed to be a diagonal matrix of functions. \\
\indent These definitions make sense. Indeed, 
\begin{enumerate}
\item Equation \ref{eq:defofuforLS}: Consider the continuity path (in $0\leq s\leq 1$) $$\Delta U=(1-s)(U-f+\Delta f)+sL^{-1}_{-e^f \ln(g)+\frac{1}{r}+\alpha_0(1-t)}(e^{\lambda U}a_0).$$ At $s=0$, $U=f$ is the unique solution and satisfies the first condition in \ref{eq:defofuforLS}. Since $L^{-1}$ is increasing and smooth, the linearisation is clearly invertible. Hence, openness holds. If $t_i\rightarrow t$ and $U_i\rightarrow U$ in $C^{2,\gamma}$, then the first condition in \ref{eq:defofuforLS} is clearly met. Therefore, we simply need to exhibit a priori estimates for $U$. By the maximum principle, $U\leq C$. Thus $\vert \Delta U\vert \leq C$ and hence $\Vert U-\fint U\Vert_{W^{2,p}}\leq C$. Suppose $U$ converges (uniformly) to $-\infty$ and $U-\fint U$ converges weakly in $W^{2,p}$ to $U_{\infty}$. Then since $L^{-1}$ is still continuous even when the first condition in \ref{eq:defofuforLS} degenerates, $\Delta U_{\infty}=e^f \lambda_{max}(\ln(g))-\frac{1}{r}-\alpha_0(1-t)$. Comparing this to the lower bound on $\Delta f$, we arrive at a contradiction. Hence $U$ is bounded in $W^{2,p}$. Using elliptic theory, it is easily seen to be bounded in $C^{4,\gamma}$. Therefore we have a $C^{4,\gamma}$ solution even at $s=1$ (which is a priori bounded in $C^{4,\gamma}$ for all $(f,g,t) \in \mathcal{B}$).
\item Equation \ref{eq:defofvforLS}: In the case of a direct sum, the systems is trivial to solve and has a unique $C^{4,\gamma}$ that is a priori bounded independent of $f,g,t \in \mathcal{B}$.\\
\indent In the general case, uniqueness was already proven in \cite{PinDem}. While the existence result in \cite{UY} works when $\epsilon=e^f$ is a \emph{constant}, we need to prove that solutions exist even when $\epsilon$ is not a constant.  Consider the continuity path $$F^0_0+\delb (\del V V^{-1})=-e^{sf}\ln(V)\omega_0.$$
At $s=0$ a solution exists thanks to \cite{UY}. Lemma \ref{lem:UYlemmaforestimate} shows that $\Vert \ln(V) \Vert_{C^0}\leq C$. The arguments prior to Lemma \ref{lem:aprioriestimatesforuniqueness} in Section \ref{sec:Aprioriestimates} show that $V$ is controlled in $C^{4,\gamma}$. Hence we are done.
\end{enumerate}
Define the Leray-Schauder map $F(f,g,t)=(f,g)-(U,V)$. Since $U,V$ are controlled in $C^{4,\gamma}$, $(f,g,t)\rightarrow (U,V)$ is a compact map. Moreover, $F^{-1}(0) \cap \partial \mathcal{B}=\phi$ because $\Vert f \Vert+\Vert \ln(g) \Vert\leq C$, $\Delta f-e^f \ln(g)+\frac{1}{r}+(1-t)\alpha_0\geq \frac{1}{C}$ on $F^{-1}(0)$. Thus the Leray-Schauder degree $deg(\mathcal{B},F(,t),0)$ is well-defined and independent of $t$. At $t=0$ the results of Section \ref{sec:Opennessatt=0} show that the degree is $\pm 1$. Thus $F(,t)^{-1}(0)\cap \mathcal{B} \neq \phi$ for all $t$. We are done.\qed

\end{document}